\documentclass[12pt]{article}

\usepackage[T1]{fontenc}
\usepackage{avant}
\usepackage[cp1250]{inputenc}
\usepackage{amsmath,amssymb,amsthm}
\usepackage{geometry}
\usepackage{caption}
\usepackage{float}
\usepackage{graphicx}

\usepackage{makeidx}
\usepackage[matrix,arrow,curve]{xy}
\usepackage{boxedminipage}
\usepackage{epic}
\usepackage{longtable}
\usepackage{ifthen}
\usepackage{tabularx}
\usepackage{parskip}
\usepackage{verbatim}
\usepackage{multirow}

\usepackage{authblk}

\title{Proper actions on strongly regular homogeneous spaces}
\author{Maciej Boche\'nski}

\begin{document}

\newtheorem{theorem}{Theorem}
\newtheorem{proposition}{Proposition}
\newtheorem{lemma}{Lemma}
\newtheorem{definition}{Definition}
\newtheorem{example}{Example}
\newtheorem{note}{Note}
\newtheorem{corollary}{Corollary}
\newtheorem{remark}{Remark}
\newtheorem{fact}{Fact}

\maketitle{}

\begin{abstract}
Let $G/H$ be a strongly regular homogeneous space such that $H$ is a Lie group of inner type. We show that $G/H$ admits a proper action of a discrete non-virtually abelian subgroup of $G$ if and only if $G/H$ admits a proper action of a subgroup $L\subset G$ locally isomorphic to $SL(2,\mathbb{R}).$ We classify all such spaces. 
\end{abstract}

\section{Introduction}\label{sec:intro}

Recall that if $L$ is a locally compact topological group acting continuously on a locally Hausdorff topological space $M$ then this action is called {\it \textbf{proper}} if for every compact subset $C \subset M$ the set
$$L(C):=\{  g\in L \ | \ g\cdot C \cap C \neq \emptyset \}$$
is compact. If $L$ is discrete and acts properly on $M$ then we say that L acts {\it \textbf{properly discontinuously}} on $M.$ We will restrict our attention to the case where $M=G/H$ is a homogeneous space of reductive type and $L \subset G.$ In more detail we assume that $G$ is a linear connected reductive real Lie group with the Lie algebra $\mathfrak{g}$ and $H\subset G$ is a closed subgroup of $G$ with finitely many connected components. Also let $\mathfrak{h}$ be the Lie algebra of $H.$

\begin{definition}
 The subgroup $H$ is reductive in $G$ if $\mathfrak{h}$ is reductive in $\mathfrak{g},$ that is, there exists a Cartan involution $\theta $ of $\mathfrak{g}$ for which $\theta (\mathfrak{h}) = \mathfrak{h}.$
 The space $G/H$ is called the homogeneous space of reductive type.
 \label{def1}
\end{definition}

The problem of proper actions was studied for example in \cite{kul}, \cite{kob2}, \cite{kob1}, \cite{ben}, \cite{ok}, \cite{kas} and \cite{bt}. We also refer the reader to the excellent survey \cite{ko}.

In the present paper we would like to examine the connection between the following three conditions
\begin{itemize}
	\item C1 := the space $G/H$ admits a properly discontinuous action of an infinite subgroup of $G,$
	\item C2 := the space $G/H$ admits a properly discontinuous action of a non-virtually abelian infinite subgroup of $G,$
	\item C3 := the space $G/H$ admits a proper action of a subgroup $L\subset G$ locally isomorphic to $SL(2,\mathbb{R}),$
\end{itemize}
(a discrete group is non-virtually abelian if it does not contain an abelian subgroup of finite index). Notice that if $H$ is compact then every closed subgroup of $G$ acts discontinuously on $G/H.$ For example the mentioned three conditions are equivalent:
$$\text{\rm C1} \Leftrightarrow \text{\rm C2} \Leftrightarrow \text{\rm C3}.$$
However this is not the case with non-compact $H.$ Indeed let
$$\mathfrak{g}=\mathfrak{k}+\mathfrak{p}$$
be a Cartan decomposition of $\mathfrak{g}$ induced by a Cartan involution $\theta$ and choose a maximal abelian subspace $\mathfrak{a}$ of $\mathfrak{p}.$ One can show that all maximal abelian subspaces of $\mathfrak{p}$ are conjugate and therefore we can define  the \textbf{\textit{real rank}} of $\mathfrak{g}$ by
$$\text{\rm rank}_{\mathbb{R}}\mathfrak{g} := \text{\rm dim} (\mathfrak{a}).$$
In this setting we can state the famous Calabi-Markus phenomenon proved by T. Kobayashi (c.f. Corollary 4.4 in \cite{kob2})

\begin{center}
$G/H$ fulfills the condition C1 if and only if $\text{\rm rank}_{\mathbb{R}}\mathfrak{g} > \text{\rm rank}_{\mathbb{R}}\mathfrak{h}.$ 
\end{center}

Take $K\subset G$ the maximal compact subgroup corresponding to $\mathfrak{k}.$ The group $K$ acts on $\mathfrak{p}$ by the adjoint representation. Define the Weyl group $W_{g}:=N_{K}(\mathfrak{a})/Z_{K}(\mathfrak{a})$ where $N_{K}(\mathfrak{a})$ (resp. $Z_{K}(\mathfrak{a})$) denotes the normalizer (resp. centralizer) of $\mathfrak{a}$ in $K.$ Then $W_{g}$ acts on $\mathfrak{a}$ by orthogonal transformations and $W_{g}$ is isomorphic to the finite reflection group generated by the restricted root system of $\mathfrak{g}$ (see Sections \ref{sec:resroot} and \ref{sec:ref}). Let $w_{0} \in W_{g}$ be the longest element. Put
\begin{equation}
\mathfrak{b}:= \{  X\in \mathfrak{a} \ | \ -w_{0}(X)=X \}.
\label{eq1}
\end{equation}
Also let $\mathfrak{h}=\mathfrak{k}_{h}+\mathfrak{p}_{h}$ be a Cartan decomposition of $\mathfrak{h}$ with respect to $\theta |_{\mathfrak{h}}.$ Notice that we can choose a maximal abelian subspace $\mathfrak{a}_{h}$ of $\mathfrak{p}$ so that $\mathfrak{a}_{h} \subset \mathfrak{a}.$
We have the following.
\begin{theorem}[Theorem 1 in \cite{ben}]
The group $G$ contains a discrete and non-virtually abelian subgroup which acts properly discontinuously on $G/H$ if and only if for every $w$ in $W_{g}$, $w \cdot \mathfrak{a}_{h}$ does not contain $\mathfrak{b}.$
\label{ben}
\end{theorem}
For example the space $SL(3,\mathbb{R})/SL(2,\mathbb{R})$ fulfills the condition C1 but does not fulfill the condition C2 and therefore the condition C3 (see Example 1 in \cite{ben}). Furthermore, consider the following example.
\begin{example}
Take $SO(4,4)/U$ where $U$ has a Lie algebra denoted by $\mathfrak{u}$ and $\mathfrak{u}\cong \mathbb{R}^{2}+\mathfrak{sl}(2,\mathbb{R}).$ The embedding $U\hookrightarrow G$ is induced by $\mathfrak{u} \hookrightarrow \mathfrak{so}(4,4)$ and is described in the Appendix.
\label{ex1}
\end{example}
This space fulfills the condition C2 but not C3. Thus for a space of reductive type we only have
$$C1 \Leftarrow C2 \Leftarrow C3.$$

The aim of this paper is to show that one can impose a condition on ``regularity'' of the embedding $H \hookrightarrow G$ so that conditions C2 and C3 become equivalent and close to C1. Let $\mathfrak{t}$ be a maximal abelian subspace of $\mathfrak{z}_{k}(\mathfrak{a}):= \{ Y\in \mathfrak{k} \ | \ \forall_{X\in \mathfrak{a}} \ [Y,X]=0  \}.$ Then
$$\mathfrak{j}:=\mathfrak{t}+\mathfrak{a}$$
is a Cartan subalgebra of $\mathfrak{g}$ called the \textbf{\textit{split Cartan subalgebra}} (for an arbitrary $\theta$ stable Cartan subalgebra $\mathfrak{w}$ the dimension $\text{\rm dim}(\mathfrak{a}\cap \mathfrak{w})$ is called the \textit{non-compact dimension} of the Cartan subalgebra, therefore a Cartan subalgebra is split if it is of maximal non-compact dimension). Analogously define a split Cartan subalgebra $\mathfrak{j}_{[h,h]}$ of $[\mathfrak{h},\mathfrak{h}].$ We propose the following definition.
\begin{definition}
The space $G/H$ is strongly regular if $\mathfrak{j}_{[h,h]}\subset \mathfrak{j} \subset \mathfrak{n}_{g}([\mathfrak{h} ,\mathfrak{h}])$ where $\mathfrak{n}_{g}([\mathfrak{h} ,\mathfrak{h} ]):=\{  Y\in \mathfrak{g} \ | \ \forall_{X\in [\mathfrak{h} ,\mathfrak{h} ]} \ [Y,X]\in [\mathfrak{h}, \mathfrak{h} ]  \}$ is a normalizer of $[\mathfrak{h} ,\mathfrak{h} ]$ in $\mathfrak{g}.$
\end{definition}
\begin{remark}
The space $G/H$ is strongly regular if and only if $\mathfrak{j}_{[h,h]}\subset \mathfrak{j}$ (up to conjugation) and $[\mathfrak{j},[\mathfrak{h},\mathfrak{h}]]\subset \mathfrak{h}.$
\end{remark}
Examples of strongly regular spaces include (but are not limited to) spaces of parabolic type (i.e. spaces where $H$ is a semisimple part of the Levi factor of any parabolic subgroup of $G$), spaces induced by regular complex spaces (i.e. spaces where $G$ and $H$ are split real forms of $G^{\mathbb{C}}$ and its regular subgroup $H^{\mathbb{C}},$ respectively), and some k-symmetric spaces (spaces where $H$ is an isotropy subgroup induced by an automorphism of $G$ of order $k$). For more details see Section \ref{sec:strreg}.
\begin{example}
Take $a,b,c \in \mathbb{N}$ so that $0<c<a<b.$ The following spaces are strongly regular (notation is close to \cite{ov2})

$$SL(b,\mathbb{R})/SL(a,\mathbb{R}), \ Sp(b,\mathbb{R})/Sp(a,\mathbb{R}), \ SO(b,b)/SO(a,a),$$
$$  SO(b,b+1)/SO(a,a+1), \ SO(b,b)/SL(a,\mathbb{R}), \ SO(b,b+1)/SL(a,\mathbb{R}), $$ 
$$SO(a,b)/SL(c,\mathbb{R}), \ SO(2a,2b)/SO(2c,2b), \ SO(2a,2b+1)/SO(2c,2b+1), $$
$$ SU(a,b)/SU(c,b), \ SU(a,b)/SL(c,\mathbb{C}), \ SU^{\ast}(2b+2)/SU^{\ast}(2a+2), $$ 
$$Sp(a,b)/SU^{\ast}(2c+2), \ E_{8}^{VIII}/SO(6,6), \ E_{8}^{IX}/E_{7}^{VII}, $$
$$ E_{6}^{III}/SU(1,6), \ F_{4}^{I}/Sp(3,\mathbb{R}), \ F_{4}^{I}/SL(3,\mathbb{R}).$$
\end{example}

Notice that the notion of strong regularity is a real analogue of the definition of regular subgroup in the complex case.

Define the \textbf{\textit{a-hyperbolic rank}} of a real reductive Lie algebra as
$$\text{\rm rank}_{a-hyp}\mathfrak{g}:=\text{\rm dim}\mathfrak{b},$$
where $\mathfrak{b}$ is given by (\ref{eq1}). The a-hyperbolic rank can be easily calculated using Table \ref{tab1} (see Section \ref{sec:ahyp}). 
\begin{definition}
We say that $G$ is of inner type if $\text{\rm rank}_{a-hyp}\mathfrak{g}=\text{\rm rank}_{\mathbb{R}}\mathfrak{g}.$
\end{definition}
\begin{remark}
The group $G$ is of inner type if and only if $\text{\rm rank}\mathfrak{k}=\text{\rm rank}\mathfrak{g}.$
\end{remark}

The main result of this paper is stated in the following theorem.
\begin{theorem}
Let $G/H$ be a strongly regular space such that $H$ is of inner type. Then
$$\text{\rm C2} \Longleftrightarrow \text{\rm C3} \Longleftrightarrow \text{\rm rank}_{a-hyp}\mathfrak{g}> \text{\rm rank}_{a-hyp}\mathfrak{h}.$$
\label{twg}
\end{theorem}
\begin{corollary}
If $G/H$ is a strongly regular homogeneous space such that $G$ and $H$ are of inner type then
$$\text{\rm C1} \Longleftrightarrow \text{\rm C2} \Longleftrightarrow \text{\rm C3} \Longleftrightarrow \text{\rm rank}_{\mathbb{R}}\mathfrak{g}> \text{\rm rank}_{\mathbb{R}}\mathfrak{h}.$$
\end{corollary}

Taking into account the Calabi-Markus phenomenon we obtain the following characterization of strongly regular spaces $G/H$ for which $H$ is of inner type
\begin{itemize}
	\item $\text{\rm C1} \Longleftrightarrow \text{\rm rank}_{\mathbb{R}}\mathfrak{g}> \text{\rm rank}_{\mathbb{R}}\mathfrak{h},$
	\item $\text{\rm C2} \Longleftrightarrow \text{\rm rank}_{a-hyp}\mathfrak{g}> \text{\rm rank}_{\mathbb{R}}\mathfrak{h},$
	\item $\text{\rm C3} \Longleftrightarrow \text{\rm rank}_{a-hyp}\mathfrak{g}> \text{\rm rank}_{\mathbb{R}}\mathfrak{h}.$
\end{itemize}

\noindent
Notice that the assumption of $H$ being of inner type in Theorem \ref{twg} is important. Consider the space from Example \ref{ex1}. The space $SO(4,4)/U$ fulfills the condition of strong regularity: $\mathfrak{j}_{[u,u]}\subset \mathfrak{j}_{so(4,4)} \subset \mathfrak{n}_{so(4,4)}([\mathfrak{u} ,\mathfrak{u}]).$ But
$$1=\text{\rm rank}_{a-hyp}\mathfrak{u}< \text{\rm rank}_{\mathbb{R}}\mathfrak{u}=3.$$

\section{Preliminaries}

In this section we introduce a notation and properties used in the proof of Theorem \ref{twg}. We retain the notation from the previous section.

\subsection{Restricted roots and the Weyl group}\label{sec:resroot}

For more information on this subject please refer to Section 4 in Chapter 4 of \cite{ov}.

Let $\mathfrak{g}^{\mathbb{C}}$ and $\mathfrak{j}^{\mathbb{C}}$ be complexifications of the reductive Lie algebra $\mathfrak{g}$ and the split Cartan subalgebra $\mathfrak{j}=\mathfrak{t}+\mathfrak{a} \subset \mathfrak{g},$ respectively. Denote by $\Delta_{gc}$ the system of roots for $\mathfrak{g}^{\mathbb{C}}$ with respect to $\mathfrak{j}^{\mathbb{C}}.$ Then
$$\mathfrak{g}^{\mathbb{C}}=\mathfrak{j}^{\mathbb{C}}+\sum_{\beta \in \Delta_{gc}}\mathfrak{g}^{c}_{\beta}.$$
Consider a subspace
$$\mathfrak{j}^{\mathbb{C}}(\mathbb{R}):=i\mathfrak{t}+ \mathfrak{a} \subset \mathfrak{j}^{\mathbb{C}}$$
and let $\mathfrak{j}^{\mathbb{C}}(\mathbb{R})^{\ast}$ be the dual space with respect to the Killing form of $\mathfrak{g}^{\mathbb{C}}$. Denote by
\begin{equation}
r:\mathfrak{j}^{\mathbb{C}}(\mathbb{R})^{\ast} \rightarrow \mathfrak{a}^{\ast}
\label{eq2}
\end{equation}
the restriction map and define
$$\Sigma_{g} \cup \{ 0 \} := r(\Delta_{gc} \cup \{ 0 \} ).$$
Then
\begin{equation}
\mathfrak{g}=\mathfrak{j}+\sum_{\alpha \in \Sigma_{g}}\mathfrak{g}_{\alpha}, \ \mathfrak{g}_{\alpha}=\mathfrak{g} \cap \sum\limits_{\substack{\beta \in \Delta_{gc} \\ r(\beta)=\alpha}}\mathfrak{g}_{\beta}^{c}
\label{eq35}
\end{equation}
is a root space decomposition for $\mathfrak{g}.$ It follows that for $\mathfrak{g}_{0} := \mathfrak{j}$ and $\alpha_{1},\alpha_{2} \in \Sigma_{g} \cup \{ 0 \}$
\begin{equation}
[\mathfrak{g}_{\alpha_{1}},\mathfrak{g}_{\alpha_{2}}] \subset \mathfrak{g}_{\alpha_{1} +\alpha_{2}}, \ \ \text{\rm where} \ \ \mathfrak{g}_{\alpha_{1} +\alpha_{2}} \neq \{ 0 \} \ \ \text{\rm iff} \ \ \alpha_{1}+\alpha_{2} \in \Sigma_{g} \cup \{ 0 \}.
\label{eq3}
\end{equation}

The Weyl group $W_{g}$ of $\mathfrak{g}$ is the finite group of orthogonal transformations of $\mathfrak{a}$ generated by reflections by hyperplanes $C_{\alpha}:=\{ X\in \mathfrak{a} \ | \ \alpha (X)=0  \}$ for $\alpha \in \Sigma_{g}.$ One can prove the following.

\begin{proposition}[Proposition 4.2, Ch. 4 in \cite{ov}]
The group $W_{g}$ coincides with the group of transformations induced by automorphisms $Ad_{k}$ $(k\in N_{K}(\mathfrak{a}))$ and also with the group of transformations induced by automorphisms $Ad_{g}$ $(g\in N_{G}(\mathfrak{a})).$ Therefore
$$W_{g}\cong N_{K}(\mathfrak{a})/Z_{K}(\mathfrak{a})\cong N_{G}(\mathfrak{a})/Z_{K}(\mathfrak{a}).$$
\end{proposition}

\subsection{Strongly regular homogeneous spaces}\label{sec:strreg}

In this subsection we would like to introduce the notion of a \textit{\textbf{strongly regular}} homogeneous space $G/H.$ For the simplicity we will assume (for this subsection only) that $\mathfrak{h}$ is semisimple. 

Recall that the complex semisimple Lie subalgebra $\mathfrak{d}$ of complex semisimple Lie algebra $\mathfrak{e}$ is called \textit{\textbf{regular}} with respect to a Cartan subalgebra $\mathfrak{s}$ of $\mathfrak{e}$ if $\mathfrak{s} \subset \mathfrak{n}_{e}(\mathfrak{d}).$ In this setting if
$$\mathfrak{e}=\mathfrak{s}+ \sum_{\beta \in \Delta_{e}}\mathfrak{e}_{\beta}$$
then
$$\mathfrak{d}=\mathfrak{s}_{d}+\sum_{\beta \in \Delta_{d}}\mathfrak{e}_{\beta},$$
where $\Delta_{d}$ is a closed and symmetric subsystem of $\Delta_{e}$ and 
$$\mathfrak{s}_{d}:=\text{\rm Span}(\{ [\mathfrak{e}_{\beta},\mathfrak{e}_{-\beta}] \ | \ \beta \in \Delta_{d} \})$$
is a Cartan subalgebra of $\mathfrak{d}.$ (see Proposition 1.1, Ch. 6 in \cite{ov}). Recall that
\begin{definition}
A subsystem $\Delta_{d}$ of $\Delta_{e}$ is closed if for every $\beta_{1} , \beta_{2} \in \Delta_{d}$
$$\beta_{1}+ \beta_{2} \in \Delta_{e} \ \Longrightarrow \ \beta_{1}+\beta_{2} \in \Delta_{d}.$$
It is symmetric if for every $\beta \in \Delta_{d}$ we have $-\beta \in \Delta_{d} .$
\end{definition}

Let $\mathfrak{c}_{d}$ be the orthogonal complement to $\mathfrak{s}_{d}$ in $\mathfrak{s}$ with respect to the Killing form. Notice that
$$\forall_{\beta \in \Delta_{d}} \ \forall_{C\in \mathfrak{c}_{d}, X\in \mathfrak{g}_{\beta}} \ [C,X]=\beta (C) X=0,$$
as $C\in C_{\beta}.$ Therefore
\begin{equation}
[\mathfrak{c}_{d},\mathfrak{d}]=0.
\label{eq77}
\end{equation}

In the real case we cannot just assume that \textit{some} Cartan subalgebra of $\mathfrak{g}$ is contained in $\mathfrak{n}_{g}(\mathfrak{h})$ because not all Cartan subalgebras in $\mathfrak{g}$ are conjugate. Therefore we modify our requirement to
$$\mathfrak{j}_{h}\subset \mathfrak{j} \subset \mathfrak{n}_{g}(\mathfrak{h})$$
(recall that $\mathfrak{j}_{h}=\mathfrak{t}_{h}+\mathfrak{a}_{h}$ and $\mathfrak{j}=\mathfrak{t}+\mathfrak{a}$ are split-Cartan subalgebras of $\mathfrak{h}$ and $\mathfrak{g},$ respectively). Also let
$$\mathfrak{g}=\mathfrak{j}+\sum_{\alpha \in \Sigma_{g}}\mathfrak{g}_{\alpha}$$
be the root space decomposition for $\mathfrak{g}.$ In this setting we obtain the following result.
\begin{lemma}
Let $G/H$ be a strongly regular homogeneous space. Then
\begin{equation}
\mathfrak{h}=\mathfrak{j}_{h}+\sum_{\alpha \in \Sigma_{h}}\mathfrak{h}_{\alpha}
\label{eq4}
\end{equation}
is a root space decomposition for $\mathfrak{h}$ such that $\Sigma_{h} \subset \Sigma_{g}$ is a subsystem and so by (\ref{eq35}) we have $\mathfrak{h}_{\alpha}\subset \mathfrak{g}_{\alpha}$ for $\alpha \in \Sigma_{h}.$ Also 
$$\mathfrak{a}_{h}=\text{\rm Span}(\{  [\mathfrak{g}_{\alpha},\mathfrak{g}_{-\alpha}] \ | \ \alpha \in \Sigma_{h} \}).$$
\label{lg}
\end{lemma}

\begin{proof}
First notice that $\mathfrak{h}^{\mathbb{C}}$ is a regular subalgebra of $\mathfrak{g}^{\mathbb{C}}$ with respect to $\mathfrak{j}^{\mathbb{C}}.$ Therefore
$$\mathfrak{h}^{\mathbb{C}}=\hat{\mathfrak{j}}^{\mathbb{C}}_{h}+\sum_{\beta \in \Delta_{hc}}\mathfrak{g}_{\beta}^{c},$$
where $\Delta_{hc}$ is a subsystem of $\Delta_{gc}$ and $\hat{\mathfrak{j}}_{h}^{\mathbb{C}} \subset \mathfrak{j}^{\mathbb{C}}$ is some Cartan subalgebra of $\mathfrak{h}^{\mathbb{C}}.$ It follows from the assumption of strong regularity that $\mathfrak{j}_{h}^{\mathbb{C}} \subset \mathfrak{j}^{\mathbb{C}}.$ Also $\hat{\mathfrak{j}}_{h}^{\mathbb{C}}$ and $\mathfrak{j}_{h}^{\mathbb{C}}$ as Cartan subalgebras of the complex semisimple lie algebra $\mathfrak{h}^{\mathbb{C}}$ are conjugate by an element $w \in H^{\mathbb{C}} \subset G^{\mathbb{C}}.$ 

Assume for the moment that $w\in W_{gc}.$ It follows that for $w\hat{\mathfrak{j}}_{h}^{\mathbb{C}}=\mathfrak{j}_{h}^{\mathbb{C}}$ we have
$$\mathfrak{h}^{\mathbb{C}} = w\mathfrak{h}^{\mathbb{C}}=w\hat{\mathfrak{j}}_{h}^{\mathbb{C}}+w\sum_{\beta \in \Delta_{hc}}\mathfrak{g}_{\beta}^{c}=\mathfrak{j}_{h}^{\mathbb{C}}+\sum_{\beta \in w(\Delta_{hc})}\mathfrak{g}_{\beta}^{c}.$$
Thus we may assume that 
$$\hat{\mathfrak{j}_{h}^{\mathbb{C}}}=\mathfrak{j}_{h}^{\mathbb{C}}$$
and therefore
$$\mathfrak{h}^{\mathbb{C}}=\mathfrak{j}_{h}^{\mathbb{C}}+\sum_{\beta \in \Delta_{hc}}\mathfrak{g}_{\beta}^{c}.$$
Now the lemma follows from the definition of the restricted root system and the restriction map (defined by (\ref{eq2})).

We only need to prove that $w\in W_{gc}.$ Let $\mathfrak{c}$ be the orthogonal complement of $\hat{\mathfrak{j}}_{h}^{\mathbb{C}}$ in $\mathfrak{j}^{\mathbb{C}}.$ We have
$$\mathfrak{j}^{\mathbb{C}}=\hat{\mathfrak{j}}_{h}^{\mathbb{C}} + \mathfrak{c}$$
and thus by (\ref{eq77})
$$w\mathfrak{j}^{\mathbb{C}}=w\hat{\mathfrak{j}}_{h}^{\mathbb{C}} + w\mathfrak{c}=\mathfrak{j}_{h}^{\mathbb{C}} + \mathfrak{c} \subset \mathfrak{j}^{\mathbb{C}}.$$
as $w\in H^{\mathbb{C}}.$ Therefore $\mathfrak{j}^{\mathbb{C}}_{h}$ and $\hat{\mathfrak{j}}_{h}^{\mathbb{C}}$ are conjugate by an element of $N_{G^{\mathbb{C}}}(\mathfrak{j}^{\mathbb{C}}).$ Without loss of generality we may assume that this element belongs to $W_{gc}$ because
$$W_{gc}\cong N_{G^{\mathbb{C}}}(\mathfrak{j}^{\mathbb{C}})/Z_{G^{\mathbb{C}}}(\mathfrak{j}^{\mathfrak{j}^{\mathbb{C}}}).$$
\end{proof}

It follows from (\ref{eq35}) and (\ref{eq3}) that if $\mathfrak{h}$ is given by (\ref{eq4}) then it fulfills the condition of strong regularity. Therefore Lemma \ref{lg} can be used to characterize strongly regular homogeneous spaces with semisimple $\mathfrak{h}.$

It also follows from Lemma \ref{lg} how one can construct some easy examples of strongly regular homogeneous spaces. For instance

\begin{enumerate}
	\item If $\mathfrak{g}$ and $\mathfrak{h}$ are split real forms of $\mathfrak{g}^{\mathbb{C}}$ and its regular subalgebra $\mathfrak{h}^{\mathbb{C}},$ respectively. Then the root system of $\mathfrak{h}^{\mathbb{C}}$ can be constructed from the extended Dynkin diagram of $\mathfrak{g}^{\mathbb{C}}.$ Since for the split real form we can identify $\Delta_{gc}$ with $\Sigma_{g}$ this yields a class of examples of strongly regular homogeneous spaces (see \cite{ov}, Chapter 6, Section 1.1).
	\item Also if $\mathfrak{o}$ is a centralizer in $\mathfrak{g}$ of any subspace in $\mathfrak{a}$ then $\mathfrak{h}:=[\mathfrak{o},\mathfrak{o}]$ induces a strongly regular homogeneous space. In this case the Satake diagram of $\mathfrak{h}$ can be obtain from the Satake diagram of $\mathfrak{g}$ by deleting any subset (different subsets correspond to different centralizers) of white vertexes and all connections and arrows which lead to this subset of white vertexes (see Theorem 1.6 and comments after, Ch. 6 in \cite{ov}).
\end{enumerate}

A more nontrivial example (as we may obtain $\mathfrak{h}_{\alpha} \subsetneq \mathfrak{g}_{\alpha}$) can be constructed as follows. Take a connected semisimple real Lie group $M$ of inner type with the Lie algebra $\mathfrak{m}$. Then there exists a closed subgroup $\tilde{U} \subset M$ locally isomorphic to $SL(2,\mathbb{R}) \times ... \times SL(2,\mathbb{R})=[SL(2,\mathbb{R})]^{\text{\rm rank}\mathfrak{m}}.$ The Lie algebra of $\tilde{U}$ is induced by a set of $\text{\rm rank}\mathfrak{m}$ orthogonal restricted roots generating commuting $\mathfrak{sl}(2,\mathbb{R})$ subalgebras (such set of roots exists since $M$ is of inner type). The resulting homogeneous space $M/\tilde{U}$ is strongly regular.

Also some \textit{k-symmetric spaces} are strongly regular. For instance it is easy to see that a 3-symmetric space $SO(n,n+3)/U(1)\times SO(n,n+1),$ $n\geq 1$ is strongly regular (see \cite{g1} for a classification of 3-symmetric spaces).

\subsection{Finite groups of reflections}\label{sec:ref}

For a more detailed treatment of this subject please refer to Chapter I.1 in \cite{hum}. Recall that we denote by $W_{g}$ the finite reflection group generated by roots in $\Sigma_{g}.$ In more detail if $\alpha \in \Sigma_{g}$ then let $s_{\alpha}$ be a reflection in $\mathfrak{a}$ through a hyperplane
\begin{equation}
C_{\alpha}= \{ X\in \mathfrak{a} \ | \ \alpha (X)=0 \}
\label{eq5}
\end{equation}
with respect to the scalar product given by the restriction of the Killing form of $\mathfrak{g}$ to $\mathfrak{a}$. Let $\Pi_{g} \subset \Sigma_{g}$ be a subset of simple roots. Then $W_{g}$ is generated by \textit{simple reflections} $s_{\alpha}$ for $\alpha \in \Pi_{g}.$ We also have $W_{g}(\Sigma_{g})=\Sigma_{g}.$

For an arbitrary $w\in W_{g}$ let $r$ be the smallest number such that $w=s_{\alpha_{1}}...s_{\alpha_{r}}$ is a product of simple reflections.

\begin{definition}
The number $r$ is called the length of $w$ and is denoted by $l(w):=r.$ By definition $l(e)=0.$
\end{definition}
Using the above definition we can describe a unique element of the Weyl group called the longest element of $W_{g}.$

\begin{definition}[see section 1.8 in \cite{hum}]
The longest element of the Weyl group is the element $w_{0}\in W_{g}$ such that $l(w_{0}w)=l(w_{0})-l(w)$ for any $w\in W_{g}.$
\end{definition}
As a consequence we obtain the following property.
\begin{fact}
The element $w_{0}$ is an involution (that is $w_{0}^{2}=id$). Also $w_{0}(\Pi_{g})=-\Pi_{g}.$
\end{fact}
Define \textbf{\textit{the fundamental Weyl chamber}} $\mathfrak{a}^{+}$ as
$$\mathfrak{a}^{+}=\{ X\in \mathfrak{a} \ | \ \forall_{\alpha \in \Pi_{g}} \ \alpha (X) \geq 0   \}.$$
Then $w_{0}(\mathfrak{a}^{+})=-\mathfrak{a}^{+}$ and
\begin{theorem}[c.f. Theorem 1.12 in \cite{hum}]
The subset $\mathfrak{a}^{+}$ is a fundamental domain of the action of $W_{g}$ on $\mathfrak{a}.$ Also
\begin{enumerate}
	\item if $wX_{1}=X_{2}$ for some $X_{1},X_{2}\in\mathfrak{a}^{+}$ and $w\in W_{g}$ then $X_{1}=X_{2}.$ Moreover, if $X_{3}\in \text{\rm Int}(\mathfrak{a}^{+}):=\{ X\in \mathfrak{a} \ | \ \forall_{\alpha \in \Pi_{g}} \ \alpha (X) > 0  \}$ then the isotropy group of $X_{3}$ is trivial.
	\item if $U\subset V$ is a subset then the subgroup $W_{U}\subset W_{g}$ fixing $U$ pointwise is generated by those reflections $s_{\alpha}$ which are contained in $W_{U}.$
\end{enumerate}
\label{weyl}
\end{theorem}
We will also need the following lemma.
\begin{lemma}
Let $e\neq w\in W_{g}$ and $S\subset \mathfrak{a}$ be a set that $w$ fixes pointwise. Then there exists $\alpha \in \Sigma_{g}$ such that $S\subset C_{\alpha}$ (where $C_{\alpha}$ is defined by (\ref{eq5})).
\label{ref}
\end{lemma}

\begin{proof}
Let $W_{S}\subset W_{g}$ be the subgroup fixing $S$ pointwise. Then $w\in W_{S}$ so $W_{S}$ is nontrivial. It follows from Theorem \ref{weyl} that there exists a reflection in $W_{S}.$ Take $\alpha$ to be a root generating this reflection.
\end{proof}

\subsection{The a-hyperbolic rank}\label{sec:ahyp}

A-hyperbolic dimensions of simple real Lie algebras can be calculated using data in Table \ref{tab1} (for a detailed description how to calculate the a-hyperbolic rank of a simple Lie algebra please refer to \cite{bt}).

\begin{center}
 \begin{table}[h]
 \centering
 {\footnotesize
 \begin{tabular}{| c | c | c |}
   \hline
   \multicolumn{3}{|c|}{ \textbf{\textit{a-hyperbolic ranks of simple Lie algebras}}} \\
   \hline                        
   $\mathfrak{g}$ & $\text{rank}_{a-hyp} (\mathfrak{g})$ & $\text{rank}_{\mathbb{R}} (\mathfrak{g})$ \\
   \hline
   $\mathfrak{sl}(2k,\mathbb{R})$  & k &  2k-1 \\
   {\scriptsize $k\geq 2$} & & \\
   \hline
   $\mathfrak{sl}(2k+1,\mathbb{R})$  & k & 2k \\
   {\scriptsize $k\geq 1$} & & \\
   \hline
   $\mathfrak{su}^{\ast}(4k)$  & k & 2k-1 \\
   {\scriptsize $k\geq 2$} & & \\
   \hline
   $\mathfrak{su}^{\ast}(4k+2)$  & k & 2k \\
   {\scriptsize $k\geq 1$} & & \\
   \hline
   $\mathfrak{so}(2k+1,2k+1)$  & 2k & 2k+1 \\
   {\scriptsize $k\geq 2$} & &  \\
   \hline
	 $\mathfrak{e}_{6}^{\text{I}}$ & 4 & 6 \\
	 \hline
   $\mathfrak{e}_{6}^{\text{IV}}$  & 1 & 2 \\
   \hline  
 \end{tabular}
 }
\captionsetup{justification=centering}
 \caption{
 The table contains all simple real Lie algebras $\mathfrak{g},$ for which $\text{rank}_{\mathbb{R}}(\mathfrak{g}) \neq \text{rank}_{a-hyp}(\mathfrak{g})$ (notation is close to Table 9, page 312 of \cite{ov2}).}
 \label{tab1}
 \end{table}
\end{center}

Table \ref{tab1} can also be used to calculate the a-hyperbolic rank of a reductive real Lie algebra.

\begin{enumerate}
	\item The a-hyperbolic rank of a semisimple Lie algebra equals the sum of a-hyperbolic ranks of all its simple parts.
	\item The a-hyperbolic rank of a reductive Lie algebra equals the a-hyperbolic rank of its derived subalgebra.
\end{enumerate}

There is a close relation between $\mathfrak{b}$ and the set of antipodal hyperbolic orbits in $\mathfrak{g}.$ Recall that an element $X \in \mathfrak{g}$ is called {\it \textbf{hyperbolic}} if $X$ is semisimple (that is, $ad_{X}$ is diagonalizable) and all eigenvalues of $ad_{X}$ are real.
\begin{definition}
An adjoint orbit $GX:=Ad(G)(X)=\{ gX:=Ad_{g}X \ | \ g\in G \}$ is said to be hyperbolic if $X$ (and therefore every element of $GX$) is hyperbolic. An orbit $GY$ is antipodal if $-Y\in GY$ (and therefore for every $Z\in GY,$ $-Z\in GY$).
\end{definition} 
We have the following lemma.

\begin{lemma}
For every hyperbolic orbit $GX$ in $\mathfrak{g}$ the set $GX \cap \mathfrak{a}$ is a single $W_{g}$ orbit in $\mathfrak{a}$.
Furthermore hyperbolic orbit $GY$ in $\mathfrak{g}$ is antipodal if and only if it meets $\mathfrak{b},$ that is
$$\mathfrak{b} \cap GY$$
is nonempty.
\label{lma}
\end{lemma}

\begin{proof}
It is standard to show that $GX \cap \mathfrak{a}$ is a single $W_{g}$ orbit (see for example Proposition 2.4 in \cite{kos}). If $X \in \mathfrak{b}$ then the longest element $w_{0}\in W_{g}$ takes $X$ to $-X$ so $GX$ is antipodal. Conversely if $GX$ is antipodal then it follows from Theorem \ref{weyl} that there exists $Y\in GX$ such that $Y\in \mathfrak{a}^{+}.$ Also $-Y\in -\mathfrak{a}^{+}$ and so $w_{0}(-Y)\in \mathfrak{a}^{+}.$ Again it follows from Theorem \ref{weyl} that $w_{0}(-Y)=Y.$ Thus
$$w_{0}Y=-Y$$
and so $Y\in \mathfrak{b}.$
\end{proof}

\subsection{Criterion of proper actions}\label{sec:prop}

Let $L\subset G$ be a subgroup of reductive type in $G.$ After conjugation by an element of $G$ we may assume that the Lie algebra $\mathfrak{l}$ of $L$ has a split Cartan subalgebra
$$\mathfrak{j}_{l}=\mathfrak{t}_{l}+\mathfrak{a}_{l}$$
such that $\mathfrak{a}_{l}\subset \mathfrak{a}.$ Then the following criterion of proper actions holds.

\begin{theorem}[Theorem 4.1 in \cite{kob2}] The following conditions are equivalent
\\ (i) $H$ acts on $G/L$ properly,
\\ (ii) $L$ acts on $G/H$ properly,
\\ (iii) $\mathfrak{a}_{h} \cap W_{\mathfrak{g}} \mathfrak{a}_{l} = \{ 0 \}.$
\label{tkob}
\end{theorem}

Now assume that $L$ is locally isomorphic to $SL(2,\mathbb{R}).$ Then $\mathfrak{a}_{l}$ is 1-dimensional. On the other hand $\mathfrak{l}\cong \mathfrak{sl}(2,\mathbb{R})$ and so $\mathfrak{l}$ is generated by three vectors $(H,E,F)$ of $\mathfrak{g}$ (called \textbf{\textit{sl(2,R)-triple}}) such that $H$ is semisimple and
$$[H,E]=2E, \ \ [H,F]=-2F \ \text{\rm and} \ [E,F]=H.$$
Thus $\mathfrak{a}_{l}=\mathbb{R}H$ and we can reformulate the above criterion using Lemma \ref{lma}.

\begin{lemma}
$L$ acts properly on $G/H$ if and only if $W_{g}H \cap \mathfrak{a}_{h} = \emptyset .$
\label{proper}
\end{lemma}

\section{Proof of Theorem \ref{twg}}\label{sec:proof}

First notice the following.

\begin{lemma}
Let $B,V_{1},...,V_{n}$ be subspaces of a finite dimensional Euclidean space $V$ and let
$$B\subset \bigcup_{k=1}^{n} V_{k}.$$
Then there exists $j, \ 1\leq j \leq n$ such that $A\subset V_{j}.$
\label{lemlem}
\end{lemma}
\begin{proof}
We will use induction on $n.$ If $n=1$ then $B \subset V_{1}.$ So assume that the lemma is valid for the sum of $n-1$ linear subspaces of $V.$ Let
$$B\subset \bigcup_{k=1}^{n} V_{k}.$$ 
If there exists $i, \ 1\leq i \leq n$ such that
\begin{equation}
B \subset \bigcup\limits_{k=1,k\neq i}^{n} V_{k}
\label{eqq1}
\end{equation}
then we can apply induction to (\ref{eqq1}). If it is not the case then we can find nonzero vectors $X,Y \in B$ such that
$$X\in \bigcup\limits_{k=1,k\neq i}^{n} V_{k}, \ X\notin V_{i}, \ Y\notin \bigcup\limits_{k=1,k\neq i}^{n} V_{k}, \ Y\in V_{i}.$$
But then for $a\neq 0$ we have $aX+bY \notin V_{i}$ and $aX+bY \in B \subset \bigcup\limits_{k=1}^{n} V_{k}.$ Therefore
$$\{ aX+bY \ | \ a\neq 0 \} \subset \bigcup\limits_{k=1,k\neq i}^{n} V_{k}.$$
After taking closures of this sets we obtain
$$\text{\rm Span}(X,Y) \subset \bigcup\limits_{k=1,k\neq i}^{n} V_{k} \ \Rightarrow \ Y\in \bigcup\limits_{k=1,k\neq i}^{n} V_{k}.$$
A contradiction.
\end{proof}

Let $w_{0}$ and $w_{0}^{h}$ be the longest elements for $W_{g}$ and $W_{h},$ respectively. Define $\mathfrak{b}$ and $\mathfrak{b}_{h}$ as in (\ref{eq1}). We can assume that 
\begin{equation}
\mathfrak{b}_{h}\subset \mathfrak{b}.
\label{eq8}
\end{equation}
Indeed if $HX$ is an antipodal hyperbolic orbit in $\mathfrak{h},$ $X \in \mathfrak{b}_{h}\subset \mathfrak{a}_{h} \subset \mathfrak{a},$ then $GX$ is an antipodal hyperbolic orbit in $\mathfrak{g}.$ Therefore there exists $g\in G$ such that
$$gX\in \mathfrak{b} \subset \mathfrak{a}.$$
It follows from Lemma \ref{lma} that there exists $w\in W_{g}$ for which
$$wX \in \mathfrak{b} \ \Rightarrow X\in w^{-1}\mathfrak{b}.$$
We obtain
$$\mathfrak{b}_{h} \subset \bigcup_{w\in W_{g}} w\mathfrak{b}.$$
Since $W_{g}$ is finite it follows from Lemma \ref{lemlem} that there exists $w_{1}\in W_{g}$ with the following property: $\mathfrak{b}_{h} \subset w_{1}\mathfrak{b}.$ After conjugating $\mathfrak{h}$ by $w_{1}$ we obtain (\ref{eq8}).

\textbf{Step 1}: $C2\Rightarrow C3$

Since $\mathfrak{b}_{h}\subset \mathfrak{b}$ it follows from Theorem \ref{ben} that $\text{\rm dim}\mathfrak{b}_{h}< \text{\rm dim}\mathfrak{b}.$ Also, since $G/H$ is a strongly regular space we can assume that $\Sigma_{h} \subset \Sigma_{g}$ (see Lemma \ref{lg}). Thus $W_{h}\subset W_{g}$ and $w_{0}^{h}\neq w_{0}.$

Since $H$ is of inner type we have $\mathfrak{a}_{h}=\mathfrak{b}_{h}.$ Therefore
$$\forall_{X\in \mathfrak{a}_{h}} \ w_{0}w_{0}^{h}(X)=w_{0}(-X)=X$$
and so $\mathfrak{a}_{h}$ is a set that $w_{0}w_{0}^{h}\in W_{g}$ fixes pointwise. It follows from Lemma \ref{ref} that we can choose $\alpha \in \Sigma_{g}$ so that
\begin{equation}
\alpha (\mathfrak{a}_{h})\equiv 0.
\label{eq10}
\end{equation}

For every root $\alpha \in \Pi_{g}$ construct a sl(2,R)-triple $( H_{\alpha}, E_{\alpha}, F_{-\alpha} )$ so that $E_{\alpha} \in \mathfrak{g}_{\alpha},$ $F_{-\alpha} \in \mathfrak{g}_{-\alpha}$ and $H_{\alpha}\in [\mathfrak{g}_{\alpha},\mathfrak{g}_{-\alpha}].$ Define $H \in \mathfrak{a}$ by: $\alpha (H)=2$ for every $\alpha \in \Pi_{g}$ and $H\in [\mathfrak{g},\mathfrak{g}].$ Since $\{ H_{\alpha} \ | \ \alpha \in \Pi_{g} \}$ spans $\mathfrak{a} \cap [\mathfrak{g},\mathfrak{g}]$ we can find $\{ a_{\alpha} \ | \ \alpha \in \Pi_{g} \} \subset \mathbb{R}$ so that
$$H= \sum_{\alpha \in \Pi_{g}} a_{\alpha}H_{\alpha}.$$
Set
$$E:=\sum_{\alpha \in \Pi_{g}}E_{\alpha}, \ \ F:=\sum_{\alpha \in \Pi_{g}} a_{\alpha}F_{-\alpha}.$$
Then $$[H,E] = \sum\limits_{\alpha \in \Pi_{g}} \alpha (H) E_{\alpha}=2E,$$
$$[H,F] = \sum\limits_{\alpha \in \Pi_{g}} a_{\alpha}(-\alpha (H)) F_{-\alpha}=-2F,$$
\begin{equation}
[E,F] = \sum\limits_{\alpha ,\beta \in \Pi_{g}} a_{\beta}[E_{\alpha},F_{-\beta}]=\sum\limits_{\alpha \in \Pi_{g}} a_{\alpha}[E_{\alpha},F_{-\alpha}]=\sum\limits_{\alpha \in \Pi_{g}} a_{\alpha} H_{\alpha}=H,
\label{eq11}
\end{equation}
and (\ref{eq11}) follows from (\ref{eq3}) since the difference of two simple roots is never a root. Therefore $( H,E,F  )$ defines a sl(2,R)-triple in $\mathfrak{g}.$

We will show that the subgroup $L\subset G$ induced by $\mathfrak{l}=\text{\rm Span}(H,E,F)\cong \mathfrak{sl}(2,\mathbb{R})$ acts properly on $G/H.$ By Lemma \ref{proper} we only need to show that
$$W_{g}H \cap \mathfrak{a}_{h} = \emptyset .$$
Assume that there exists $w_{2}\in W_{g}$ such that
$$w_{2}H \in \mathfrak{a}_{h} \ \Rightarrow \ H\in w_{2}^{-1}\mathfrak{a}_{h}.$$
By (\ref{eq10}) $(w_{2}^{-1}\alpha) (H)=0$ and so $s_{w_{2}^{-1}\alpha}H=H$. By construction of $(H,E,F)$ we have $H \in \text{\rm Int}(\mathfrak{a}^{+}).$ But Theorem \ref{weyl} implies that the isotropy group of $H$ is trivial - a contradiction.

\textbf{Step 2}: $\text{\rm rank}_{a-hyp}\mathfrak{g}\leq \text{\rm rank}_{a-hyp}\mathfrak{h} \ \Rightarrow \ \neg C3$

It follows from (\ref{eq8}) and the equality $\mathfrak{a}_{h}=\mathfrak{b}_{h}$ that in this case $\mathfrak{b}=\mathfrak{a}_{h}.$ By Theorem $\ref{ben}$ the space $G/H$ does not fulfill the condition $C2.$ Therefore it does not fulfill the condition $C3.$

\textbf{Step 3}: $\text{\rm rank}_{a-hyp}\mathfrak{g} > \text{\rm rank}_{a-hyp}\mathfrak{h} \ \Rightarrow \ C2$

It follows that $\text{\rm dim}\mathfrak{b} > \text{\rm dim}\mathfrak{a}_{h}.$ Thus for any $w\in W_{g},$ $w\mathfrak{a}_{h}$ does not contain $\mathfrak{b}.$ Our claim follows from Theorem \ref{ben}.

\section{Appendix}\label{sec:appen}

Please recall that the split Cartan subalgebra of $\mathfrak{so}(4,4)$ can be identified with the Euclidean space $\mathbb{R}^{4}.$ Let $\{ e_{1}, e_{2}, e_{3}, e_{4} \}$ be a standard orthonormal basis of $\mathbb{R}^{4}.$ Denote by $(\mathbb{R}^{4})^{\ast}$ the dual space of $\mathbb{R}^{4}.$ Then
$$\Sigma_{so(4,4)}=\{ \pm e_{i} \pm e_{j} \ | \ 1\leq i < j \leq 4 \} \subset (\mathbb{R}^{4})^{\ast}.$$
The Weyl group $W_{so(4,4)}$ acts on $\mathbb{R}^{4}$ by permuting coordinates of a vector and changing the sign of an even number of coordinates of a vector.
As is shown in Example 5.3.7 in \cite{cmg} semisimple elements of all (up to conjugation) sl(2,R)-triples in $\mathfrak{so}(4,4)$ are given in the first column of Table \ref{tab2}. The second column of this table shows a transformation of a given semisimple element (denoted by $w(H)$) by some element of the Weyl group $W_{so(4,4)}.$ Take $a=(3,1,0,2),$ $b=(2,0,0,1)$ and $c=(0,0,1,0).$ The third column of Table \ref{tab2} presents $w(H)$ as a linear combination of vectors $a,b,c.$ 

\begin{center}
 \begin{table}[h]
 \centering
 {\footnotesize
 \begin{tabular}{| c | c | c |}
   \hline                    
   Semisimple elements $W$ & $w(H)$ & Linear combination of $a,b,c$ \\
   \hline
   $(6,4,2,0)$  & $(6,2,0,4)$ &  $2a$ \\
   \hline
   $(4,2,2,0)$  & $(4,0,2,2)$ & $2(b+c)$ \\
   \hline
   $(3,3,1,1)$  & $(-3,1,3,-1)$ & $a-3b+3c$ \\
   \hline
   $(3,3,1,-1)$  & $(-3,1,-3,-1)$ & $a-3b-3c$ \\
   \hline
   $(4,2,0,0)$  & $(4,0,0,2)$ & $2b$ \\
   \hline
	 $(2,1,1,0)$ & $(2,0,1,1)$ & $b+c$ \\
	 \hline
   $(1,1,1,1)$  & $(1,1,1,1)$ & $a-b+c$ \\
   \hline  
	  $(1,1,1,-1)$  & $(1,1,-1,1)$ & $a-b-c$ \\
   \hline  
	  $(2,0,0,0)$  & $(0,0,2,0)$ & $2c$ \\
   \hline  
	  $(1,1,0,0)$  & $(-1,1,-0,0)$ & $a-2b$ \\
   \hline  
	  $(0,0,0,0)$  & $(0,0,0,0)$ & $0$ \\
   \hline  
 \end{tabular}
 }
\captionsetup{justification=centering}
 \caption{
 Semisimple elements of nilpotent orbits in $\mathfrak{so}(4,4).$
 }
 \label{tab2}
 \end{table}
\end{center}

 Therefore any $W_{g}$ orbit of a semisimple element of any sl(2,R) triple in $\mathfrak{so}(4,4)$ meets $\mathfrak{u}_{1}:=\text{\rm Span(a,b,c)}.$ Take the root $\alpha:=-e_{1}+e_{2}$ and let $\mathfrak{u}^{+},$ $\mathfrak{u}^{-}$ be root spaces corresponding to $\alpha$ and $-\alpha,$ respectively. We have $[\mathfrak{u}^{+},\mathfrak{u}^{-}]=\text{\rm Span}(-1,1,0,0)\subset \mathfrak{u}_{1}.$ Put
$$\mathfrak{u}:=\mathfrak{u}_{1}+\mathfrak{u}^{+}+\mathfrak{u}^{-}.$$
One easily sees that $\mathfrak{u}$ is a reductive subalgebra of $\mathfrak{so}(4,4)$ isomorphic to $\mathbb{R}^{2}+\mathfrak{sl}(2,\mathbb{R}).$ Moreover
$$\text{\rm rank}_{\mathbb{R}}\mathfrak{u}=3=\text{\rm dim}\mathfrak{u}_{1} \ \text{\rm and} \ \text{\rm rank}_{a-hyp}\mathfrak{u}=1,$$
as $\mathfrak{u}_{1}$ is a split Cartan subalgebra of $\mathfrak{u}.$ Let $U\subset SO(4,4)$ be the corresponding closed subgroup. It follows from Lemma \ref{proper} and the construction of $\mathfrak{u}$ that $G/U$ does not admit proper action of $SL(2,\mathbb{R}).$ Also
$$\text{\rm rank}_{a-hyp}\mathfrak{so}(4,4)=4=\text{\rm dim} \mathfrak{b}_{so(4,4)}> \text{\rm dim}\mathfrak{u}_{1}$$
therefore it follows from Theorem \ref{ben} that $G/U$ admits a properly discontinuous action of an infinite non-virtually abelian subgroup of $G.$

Department of Mathematics and Computer Science

University of Warmia and Mazury

S\l\/oneczna 54, 10-710, Olsztyn, Poland

e-mail: mabo@matman.uwm.edu.pl


\begin{thebibliography}{99}

\bibitem{ben} Y. Benoist,  {\it Actions propres sur les espaces homogenes reductifs},  Ann. of Math. 144 (1996),  315-347.

\bibitem{bt} M. Boche\'nski, A. Tralle, {\it Clifford-Klein forms and a-hyperbolic rank}, Int. Math. Res. Notices (2014)  DOI: 10.1093/imrn/rnu123.

\bibitem{cmg} D. H. Collingwood, W. M. McGovern, {\it Nilpotent orbits in semisimple Lie algebras}, Van Nostrand Reinhold, New York (1993).

\bibitem{g1} A. Gray and J. Wolf,  {\it Homogeneous spaces defined by Lie group automorphisms I, II,},  J. Differential Geometry 2 (1968),  77-159.

\bibitem{hum} J. E. Humphreys, {\it Reflections groups and Coxeter groups,} Cambridge University Press (1990).

\bibitem{kc} V. G. Kac, {\it Infinite dimensional Lie algebras,} Oxford University Press (1990).

\bibitem{kas} F. Kassel, {\it Deformation of proper actions on reductive homogeneous spaces}, Math. Ann. 353 (2012), 599-632.

\bibitem{ko} T. Kobayashi, T. Yoshino, {\it Compact Clifford-Klein forms of symmetric spaces revisited}, Pure Appl. Math. Quart. 1 (2005), 603-684.

\bibitem{kob1} T. Kobayashi, {\it Criterion for proper actions on homogeneous spaces of reductive groups}, J. Lie Theory 6 (1996), 147-163.

\bibitem{kob2} T. Kobayashi, {\it Proper actions on a homogeneous space of reductive type}, Math. Ann. 285 (1989), 249-263.

\bibitem{kos} B. Kostant, {\it On convexity, the Weyl group and the Iwasawa decomposition}, Ann. Sci. Ecole Norm. Sup. (4), 6 (1973) 413-455.

\bibitem{kul} R. S. Kulkarni  {\it Proper actions and pseudo-Riemannian space forms}  Adv. Math. 40 (1981), 10-51.

\bibitem{ok} T. Okuda,  {\it Classification of semisimple symmetric spaces with $SL(2, \mathbb{R})$-proper actions,}  J. Different. Geom. 94 (2013), 301-342.

\bibitem{ov2} A. L. Onishchik, E. B. Vinberg {\it Lie groups and algebraic groups,}  Springer (1990).

\bibitem{ov} A. L. Onishchik, E. B. Vinberg {\it Lie groups and Lie algebras III,}  Springer (1994).

\end{thebibliography}
\end{document}